\documentclass[10pt]{article}

\usepackage{xcolor}
\usepackage{amsthm,amsmath,amssymb}
\newtheorem{theorem}{Theorem}
\newtheorem{proposition}[theorem]{Proposition}
\newtheorem{lemma}[theorem]{Lemma}

\theoremstyle{definition}
\newtheorem{example}[theorem]{Example}
\newtheorem*{remark}{Remark}

\newcommand\qbin[3]{\left[ \begin{matrix} #1 \\ #2 \end{matrix}
                   \right]_{\displaystyle #3}}

\numberwithin{theorem}{section}
\numberwithin{equation}{section}

\title{On the $q$-factorization of power series}
\author{
\textbf{Robert Schneider}\\ Department of Mathematical Sciences\\
Michigan Technological University\\
Houghton, Michigan 49931, U.S.A.\\
\texttt{robertsc@mtu.edu} \\ \\
 \textbf{Andrew V. Sills} \\Department of Mathematical Sciences\\
 Georgia Southern University\\
Statesboro, Georgia 30460, U.S.A.\\
\texttt{asills@georgiasouthern.edu}\\ \\
\textbf{Hunter Waldron}\\
Department of Mathematical Sciences\\
Michigan Technological University\\
Houghton, Michigan 49931, U.S.A.\\
\texttt{hpwaldro@mtu.edu} 
}
\date{\today}

\begin{document}

\maketitle


\begin{abstract}
Any power series with unit constant term can be factored into an infinite product of the
form $\prod_{n\geq 1} (1-q^n)^{-a_n}$.   We give direct formulas for the exponents $a_n$ 
in terms of the coefficients of the power series, and vice versa, as sums over partitions.
As examples, we prove identities for certain partition enumeration functions.  Finally, we note $q$-analogues of our enumeration formulas.
\end{abstract}

\begin{center} {\it In honor of George Andrews and Bruce  Berndt for their 85th birthdays}\end{center}

\section*{Statements and Declarations}
The second author is a guest editor for the special issue for which this paper has been
submitted.

\section{Introduction and main results}
Many major results in number theory, analysis, and combinatorics take the form of ``a series equals a product''; in particular, the use of product--sum generating functions is a prevalent method in partition theory. 

In ~\cite[p. 98, Ex. 2]{gea76}, George Andrews considers the factorization of an ordinary power series 
with unit constant term into a $q$-product:
\begin{equation} \label{ps}
  1 + \sum_{n=1}^\infty r(n) q^n = \prod_{n=1}^\infty \frac{1}{(1-q^n)^{a_n}},
\end{equation} where the $r(n)$ and $a_n$ are real numbers.

Andrews suggests a recursion by which the $r(n)$ can be calculated from a given sequence  $a_n$, 
namely
\begin{equation} \label{rgivena}
  n r(n) = \sum_{j=1}^n r(n-j) \sum_{d\mid j} d a_d.
\end{equation}
The reverse problem, where the $r(n)$ are given and the $a_n$ are calculated recursively,
is given by David Bressoud~\cite[p. 61, Ex. 2.3.10]{b99}:
\begin{equation}\label{agivenr}
  n a_n = D_n - \underset{d < n}{\sum_{d \mid n}} d a_d,
\end{equation} 
where \[ D_m = m r(m) -  \sum_{j=1}^{m-1} D_j r(m-j). \]
Bressoud~\cite[Exs. 2.3.8 and 2.3.10 resp.]{b99} provided \emph{Mathematica} code for both~\eqref{rgivena} and~\eqref{agivenr}.
A Maple implementation of ~\eqref{agivenr} is provided by Frank Garvan~\cite{fg20} in his Maple \texttt{qseries.m} package as the
\texttt{prodmake} procedure.  
As Garvan's \texttt{prodmake} is extremely useful in searching for Rogers--Ramanujan
type identities, Shashank Kanade and Matthew Russell used it extensively in their
own work~\cite{kr15}.

A natural question to ask, then, is:  {\it can explicit (non-recursive) formulas be given, to express the sequences $a_n$ and $r(n)$ in terms of one another?} Below we present such explicit formulas in Proposition~\ref{thm1} and Theorem~\ref{thm2}.

First, let us recall   basic definitions and notations associated with partitions.
A \emph{partition} $\lambda$ of an integer $n$ is a weakly decreasing finite 
sequence of positive integers $(\lambda_1, \lambda_2, \dots, \lambda_\ell)$, where
$\lambda_1 \geq \lambda_2 \geq \cdots \geq \lambda_\ell > 0$, that sum to $n$.
Each $\lambda_i$ is called a \emph{part} of the partition $\lambda$.  The
\emph{length} $\ell = \ell(\lambda)$ of a partition $\lambda$ is the number of parts in $\lambda$.
The multiplicity $m_i = m_i(\lambda)$ of $i$ in $\lambda$ is the number of times that $i$ 
appears as a part in $\lambda$.   The notation  $\sum_{\lambda\vdash n}$ means a sum is being taken over all partitions $\lambda$ of
 $n$.

In~\cite{gm54}, G. Meinardus gave an asymptotic formula for the $r(n)$ when the
$a_n$ are all nonnegative real numbers.  
Meinardus' result was subsequently extended by H. Todt~\cite{ht11} 
to all real numbers $a_n$ provided
the $r(n)$ are increasing.   
For recent work on these asymptotics, see Bridges et al.~\cite{bbbf24}.

We now give direct formulas for $r(n)$ as a partition sum
in terms of  the exponents $a_k$, and $a_n$ as a divisor sum containing an inner sum over partitions in terms 
of the coefficients $r(k)$.

\begin{proposition} \label{thm1}
Let $r(n)$ and $a_n$ be defined as in~\eqref{ps}.  Then
\begin{equation} \label{RfromA} r(n) = \sum_{\lambda \vdash n}
  \frac{  (a_1)_{m_1} (a_2)_{m_2} \cdots }
  {m_1!\  m_2! \cdots},
 \end{equation}
 where 
 $(a)_k = (a)(a+1)(a+2)\cdots(a+k-1)$ is the usual rising factorial.  
 \end{proposition}
  
  \begin{theorem} \label{thm2}
  Let $r(n)$ and $a_n$ be defined as in~\eqref{ps}.  Then
  \begin{equation}  a_n = \frac{1}{n} \sum_{d\mid n} \mu\left( \frac{n}{d} \right) d
             \sum_{\lambda\vdash d}   
             \frac{ (-1)^{\ell - 1} (\ell-1)! \ r(1)^{m_1}  r(2)^{m_2}\cdots } 
             {  m_1!\  m_2!  \cdots } . 
     \label{AfromR}
  \end{equation}   
  
  \end{theorem}
  
  \section{Proof of Proposition~\ref{thm1}}
  Our Proposition~\ref{thm1} follows as an application of the following theorem
  of N.~J. Fine~\cite[\S22, Theorem 1]{f88}, rewritten to suit our current notation.
  \begin{theorem}[Fine]
  Let $\psi_j(q) = \sum_{k\geq 0} C_j(k) q^k,\  (j = 1,2,\dots). $ Then
  \[ \prod_{j\geq 1} \psi_j(q^j) = \sum_{n\geq 0} q^n \sum_{\lambda\vdash n}
  C_1(m_1(\lambda) ) C_2(m_2(\lambda)) \cdots .\]
  \end{theorem}
  
  All we need to do is set Fine's $\psi_j(q) := (1-q)^{-a_j}$, then 
we must have \[ C_j(k) = \frac{(a_j)_k}{k!}, \] by the Maclaurin series 
expansion of the the binomial series.   Proposition~\ref{thm1} follows immediately. \qed

\section{Some applications of Proposition~\ref{thm1}}
In this section, we present applications of Proposition \ref{thm1} to give identities for certain partition enumeration functions. 

\begin{example}
A canonical special case of~\eqref{ps} is where all $a_n = 1$, and thus $r(n) = p(n)$, the 
number of partitions of $n$.  Then~\eqref{RfromA} reduces to the trivial statement
\[p(n) = \sum_{\lambda\vdash n} 1, \] as $(1)_m = m!$.  In other words, $p(n)$ is obtained by
literally counting all partitions of $n$, one by one.  
\end{example}

\begin{example}
  Let $p_S(n)$ denote the number of partitions of $n$ in which all parts are in some subset $S\subseteq \mathbb Z^+$ of
the positive integers.
Observe that
\begin{equation} 
\sum_{n=0}^\infty p_S(n) q^n = \prod_{n\in S} \frac{1}{1-q^n}.
\end{equation}
Upon noting $(0)_0 = 1$ (the empty product) and $(0)_m = 1$ when $m$ is a positive integer,
we see that 
\begin{equation} \label{restptn}
 p_S(n) =  \sum_{\lambda\vdash n} \chi( \mbox{ all parts of $\lambda$ are in $S$ }  ),
 \end{equation}
where the characteristic function $\chi(A) = 1$ if $A$ is true and $0$ if $A$ is false.
So~\eqref{restptn} is just a brute force count of allowable partitions of size $n$ as one runs through the
unrestricted partitions of $n$.
\end{example}

\begin{example}
Recall that if $\overline{p}(n)$ denotes the number of overpartitions~\cite{cl04} of $n$, then
\begin{equation} \label{OPgf}
\sum_{n=0}^\infty \overline{p}(n) q^n = \prod_{n=1}^\infty \frac{1+q^n}{1-q^n} = 
\prod_{n=1}^\infty \frac{1}{(1-q^n)^{a_n}}
\end{equation}
where 
\[ a_n = 
\begin{cases}  
  1 &\mbox{if $n$ is even} \\
  2 &\mbox{if $n$ is odd}
\end{cases} .
\]
Thus we conclude that
\begin{equation}
  \overline{p}(n) =  \sum_{\lambda\vdash n} (m_1 + 1) (m_3 + 1) (m_5 + 1) \cdots .
\end{equation}
\end{example}

\begin{example}  Fix a positive integer $k$.
If $p_k(n)$ denotes the number of $k$-color partitions of $n$, then 
\begin{equation} \label{kColorGF}
\sum_{n=0}^\infty p_k(n) q^n =\prod_{n=1}^\infty \frac{1}{(1-q^n)^{k}}.
\end{equation}
  Thus we conclude that
\begin{equation}
p_k(n) =  \sum_{\lambda\vdash n} \frac{(k)_{m_1} (k)_{m_2} \cdots}
{m_1!\ m_2!  \cdots}  =  \sum_{\lambda\vdash n} \binom{k+m_1 - 1}{k-1}
\binom{k+m_2 - 1}{k-1}  \binom{k+m_3 - 1}{k-1} \cdots.
\end{equation}
\end{example}

\begin{example}
If $PL(n)$ denotes the number of plane partitions of $n$, then 
\begin{equation} \label{PLGF}
\sum_{n=0}^\infty PL(n) q^n =\prod_{n=1}^\infty \frac{1}{(1-q^n)^{n}}.
\end{equation}
  Thus we conclude that
\begin{equation}
PL(n) =  \sum_{\lambda\vdash n} \frac{(1)_{m_1} (2)_{m_2} (3)_{m_3} \cdots}
{m_1!\ m_2! \ m_3! \cdots} 
=  \sum_{\lambda\vdash n} \binom{m_2+1}{1} \binom{m_3+2}{2} \binom{m_4+3}{3} \cdots.
\end{equation}
\end{example}

\begin{example}
The $k$-broken diamond partitions
are defined by Andrews and P. Paule in~\cite{ap07}.   If $r_k(n)$ denotes the number of
$k$-broken diamond partitions of $n$, then
\[ \sum_{n=0}^\infty r_k(n) q^n = \prod_{n=1}^\infty \frac{1}{(1-q^n)^{a_n}}, \]
where
\[ a_n = 
\begin{cases}  
  2 &\mbox{if $n$ is even or $n\equiv (2k+1)\pmod{4k+2}$} \\
  3 &\mbox{otherwise}
\end{cases} .
\]
Thus
\begin{equation}
r_k(n) 
=\sum_{\lambda\vdash n}
\prod_{j=0}^\infty  \frac{ (2+m_{2j+1}) (1+m_{2j+2})  (1+m_{(4k+2)j+(2k+1)}) }
 {   {(2+m_{(4k+2)j+(2k+1)})}  } .
\end{equation}

\begin{remark} Further results along similar lines are noted in Appendix D of \cite{Schneider_PhD}.\end{remark}
\end{example}
  
  \section{Proof of Theorem~\ref{thm2}}
  As before, let $r(n)$ and $a_n$ be defined by~\eqref{ps} and let $b(n)$ be defined by
  \begin{equation}  \label{seriesBproductA}
      \sum_{n=1}^\infty b(n) q^n = \log \prod_{n=1}^\infty \frac{1}{(1-q^n)^{a_n}}.
  \end{equation}
If we differentiate~\eqref{seriesBproductA} with respect to $q$ and equate coefficients of
$q$, we observe that
\begin{equation}
     n b( n ) = \sum_{d\mid n} d  a_d,
\end{equation}  and thus by M\"obius inversion, we have
\begin{equation}
\label{AfromB}
   n a_n = \sum_{d\mid n} \mu\left( \frac nd\right) d\,b(d) .
\end{equation}
 Next, we need to prove a lemma.

\begin{lemma} We have that \label{lem}
\begin{multline}
n \sum_{\lambda\vdash n} \frac{ (-1)^{\ell(\lambda)-1} (\ell(\lambda)- 1)!  \ 
r(1)^{m_1(\lambda)} r(2)^{m_2(\lambda)} \cdots  }
{ m_1(\lambda)!  m_2(\lambda)!  \cdots } \\
= \sum_{j=1}^n j r(j) \sum_{\mu\vdash n-j} \frac{ (-1)^{\ell(\mu)} 
\ell(\mu)! \ r(1)^{ m_1(\mu)} r(2)^{ m_2(\mu)} \cdots }
{ m_1(\mu)! m_2(\mu)! \cdots  }.
\end{multline}
\end{lemma}
\begin{proof}
Let $\lambda$ be an arbitrary partition of size $n$.
Thus we have \begin{align*}
  n  &= \sum_{j=1}^n j m_j(\lambda)\\
      & = \sum_{j=1}^n j  \frac{m_j(\lambda)!}{(m_j(\lambda) - 1)!}
      \mbox{  (where we follow the convention that $\frac{0!}{(-1)!} := 0$)}
      \\
      & = \sum_{j=1}^n j \frac
      {m_1(\lambda)! \cdots m_{j-1}(\lambda)! m_j(\lambda)! m_{j+1}(\lambda)! \cdots
       m_n(\lambda)! }
       { m_1(\lambda)! \cdots m_{j-1}(\lambda)! (m_j(\lambda)-1)! m_{j+1}(\lambda)! \cdots
       m_n(\lambda)! } , 
\end{align*}
which implies
\[  \frac{ n(-1)^{\ell(\lambda)-1} (\ell(\lambda) - 1)! }{ m_1(\lambda)! \cdots m_n(\lambda)!}
\\ = \sum_{j=1}^n j
 \frac{ (-1)^{m_1(\lambda) + \cdots m_n(\lambda) - 1} (m_1(\lambda) + \cdots m_n(\lambda)-1)!  } { m_1(\lambda)! \cdots m_{j-1}(\lambda)! (m_j(\lambda)-1)! m_{j+1}(\lambda)! \cdots
       m_n(\lambda)!   }.
 \]
Upon multiplying both sides by $r(1)^{m_1(\lambda)}  r(2)^{m_2(\lambda)}  \cdots $
and then summing over all partitions of $n$, the result follows.
\end{proof}

Now, we have the ingredients in place to prove Theorem~\ref{thm2}. Letting $f(q) := 1 + \sum_{n=1}^\infty r(n) q^n = \prod_{n=1}^\infty (1-q^n)^{-a_n}$,
and performing logarithmic differentiation, we obtain
\[    \frac{d}{dq} \log f(q) = \frac{df}{dq} \cdot \frac{1}{f(q)}. \]
Writing each expression as a power series (see \cite[p. 3, Lemma 1]{ss22} for
the power series representation of $1/f(q)$), we obtain
\begin{align*}
\sum_{n=1}^\infty nb(n)q^{n-1} &=
 \left( \sum_{j=1}^\infty j r(j) q^{j-1} \right)
 \left(\sum_{k=0}^\infty \left[ \sum_{\lambda\vdash k }
 \frac{(-1)^\ell \ell!\ r(1)^{m_1} \cdots r(k)^{m_k}}{m_1! \cdots m_k!} \right] q^k \right)\\
 &= 
 \sum_{n=1}^\infty q^{n-1}
 \sum_{j=1}^n j r(j) \sum_{\mu\vdash n-j} \frac{ (-1)^{\ell(\mu)} 
\ell(\mu)! \ r(1)^{ m_1(\mu)} \cdots r(n-j)^{m_{n-j}(\mu)} }
{ m_1(\mu)! \cdots m_{n-j}(\mu)! }\\
&= \sum_{n=1}^\infty q^{n-1} n \sum_{\lambda\vdash n} \frac{ (-1)^{\ell(\lambda)-1} (\ell(\lambda) - 1)!  \ 
r(1)^{m_1(\lambda)} \cdots r(n)^{m_n(\lambda)} }
{ m_1(\lambda)! \cdots m_n(\lambda)! },
\end{align*}
where the last equality was proved as Lemma~\ref{lem}.

By equating powers of $q$ in the extremes, we now have $b(n)$ in terms of
$r(1), r(2), \dots, r(n)$.  Finally, apply~\eqref{AfromB}, to obtain Theorem~\ref{thm2}.

\section{Some Applications of Theorem~\ref{thm2}}
\begin{example}
Let $c(n)$ denote the number of compositions of $n$.  It is well known that 
\[ c(n) = 
\begin{cases}  
  1 &\mbox{if $n=0$} \\
  2^{n-1} &\mbox{if $n>0$}
\end{cases} ,
\]
and that 
\[ \sum_{n=0}^\infty c(n) q^n = \frac{1-q}{1-2q}, \]
but as a consequence of Theorem~\ref{thm2}, we further conclude that
\[ \sum_{n=0}^\infty c(n) q^n = \prod_{n=1}^\infty \frac{1}{(1-q^n)^{a_n}}, \] where
\begin{equation} \label{ancomps}
 a_n = \frac{1}{n} \sum_{d\mid n} \mu\left( \frac{n}{d} \right) (2^d - 1), 
 \end{equation}
i.e., the sequence $\{ a_n\}_{n=1}^\infty = (1,1,2,3,6,9,18,30,56,99,186,335, 630, \dots)$, which is OEIS sequence
A059966~\cite{oeisA059966}.   
OEIS indicates that the $a_n$ are the Lie analog of the sequence of partition numbers,
which gives the dimensions of the homogeneous polynomials with one generator in
each degree, and several other interpretations.

To see that~\eqref{ancomps} holds, apply Theorem~\ref{thm2} with $r(n) = c(n) = 2^{n-1}$ if
$n>0$ and $r(0) = c(0) = 1$.  Thus we need to show that
\begin{equation} \label{simp2d}
 d\sum_{\lambda\vdash d} \frac{(-1)^{\ell - 1} (\ell - 1)!  (2^0)^{m_1} (2^1)^{m_2} (2^2)^{m_3} \cdots}
 {m_1! \ m_2! \ m_3! \cdots } = 2^d - 1.
 \end{equation}
Apply Lemma~\ref{lem} with $n = d$ to the left side of~\eqref{simp2d}:
 \begin{align}
&  \phantom{ = }d\sum_{\lambda\vdash d} \frac{(-1)^{\ell - 1} (\ell - 1)!  (2^0)^{m_1} (2^1)^{m_2} (2^2)^{m_3} \cdots}
 {m_1! \ m_2! \ m_3! \cdots } \notag \\ & =
 \sum_{j=1}^d j \ 2^{j-1} \sum_{\mu \vdash d-j} \frac{(-1)^\ell \ \ell! \ (2^0)^{m_1} 
 (2^1)^{m_2} (2^2)^{m_3} \cdots}{m_1! \ m_2! \ m_3! \cdots } \notag \\
 & = \sum_{j=1}^d j \ 2^{j-1} \cdot \mbox{ coeff of } q^{d-j} \mbox{ in } \frac{1}{1+2q+4q^2+8q^3
 +\cdots} \label{recipr} \\
 & = \sum_{j=1}^d j \ 2^{j-1} \cdot \mbox{ coeff of } q^{d-j} \mbox{ in } \frac{1-2q}{1-q} \notag \\
 & = \sum_{j=1}^d j \ 2^{j-1} \cdot \mbox{ coeff of } q^{d-j} \mbox{ in } 1 -q -q^2 -q^3 - q^4 - \cdots
 \notag\\
 & = d\ 2^{d-1} - \sum_{j=1}^{d-1} j\ 2^{j - 1} \notag \\
 & = 2^d - 1, \notag
 \end{align}
 where \eqref{recipr} follows from the preceding line by~\cite[p. 3, Lemma 1]{ss22}, and
the last equality follows by induction on $d$.

\end{example}

\begin{example}
Let $F_n$ denote the $n$th Fibonacci number, where 
\[ F_1 = F_2 = 1, \mbox{  and } F_n = F_{n-1} + F_{n-2} \mbox{  for $n>2$ }.\]
It is well known that 
\[ \sum_{n=0}^\infty F_{n+1} q^n = \frac{1}{1-q-q^2}, \]
but as a consequence of Theorem~\ref{thm2}, we also see that
\[ \sum_{n=0}^\infty F_{n+1} q^n  = \prod_{n=1}^\infty \frac{1}{(1-q^n)^{a_n}}, \] where
\begin{equation} \label{fibexponents} 
a_n = \frac{1}{n} \sum_{d\mid n} \mu\left( \frac{n}{d} \right) L_{d}, 
\end{equation}
where $L_1 = 1,  L_2 = 3,  L_{n} = L_{n-1} + L_{n-2}$ for $n>2$ (the Lucas numbers).
The sequence $\{ a_n\}_{n=1}^\infty = 
(1, 1, 1, 1, 2, 2, 4, 5, 8, 11, 18, 25, 40, 58, 90, 135,$ $210, \dots)$ is 
OEIS sequence
A006206~\cite{oeisA006206}.

To see that~\eqref{fibexponents} holds, note that the $d=1$ and $d=2$ cases can be
easily verified by direct computation.  For $d>2$, we start by applying Theorem~\ref{thm2} with
$r(n) = F_{n+1}$, and observe that we need to show that
\begin{equation} \label{Leq}
d \sum_{\lambda\vdash d} \frac{(-1)^{\ell -1} (\ell - 1)! \ F_2^{m_1} F_3^{m_2} F_4^{m_3} \cdots }
{m_1! \ m_2! \ m_3! \cdots } = L_d.
\end{equation}

Apply Lemma~\ref{lem} with $n = d$ to the left side of~\eqref{Leq}:
 \begin{align}
&  \phantom{ = }d\sum_{\lambda\vdash d} \frac{(-1)^{\ell - 1} (\ell - 1)!  F_2^{m_1} F_3^{m_2} F_4^{m_3} \cdots}
 {m_1! \ m_2! \ m_3! \cdots } \notag \\ & =
 \sum_{j=1}^d j \ F_{j+1} \sum_{\mu \vdash d-j} \frac{(-1)^\ell \ \ell! \ F_2^{m_1} 
 F_3^{m_2} F_4^{m_3} \cdots}{m_1! \ m_2! \ m_3! \cdots } \notag \\
 & = \sum_{j=1}^d j \ F_{j+1} \cdot \mbox{ coeff of } q^{d-j} \mbox{ in }\frac{1}
 {F_1 + F_2 q + F_3 q^2 + F_4 q^3 + \cdots}  \label{recipr2} \\
 & = \sum_{j=1}^d j \ F_{j+1} \cdot \mbox{ coeff of } q^{d-j} \mbox{ in } (1-q-q^2) \notag \\
 & = \sum_{j=1}^d j \ F_{j+1} \cdot \begin{cases}  
  1 &\mbox{if $j = d$ }, \\
  -1 &\mbox{if $j =d-1$ or $d-2$, }\\
  0  &\mbox{otherwise; }
\end{cases}
 \notag\\
 & = -(d-2)F_{d-1} - (d-1)F_d + d F_{d+1}\notag \\
 & = -F_d - (d-2)(F_{d-1} + F_d) + 2 F_{d+1} + (d-2)F_{d+1} \notag \\
 & = -F_d - (d-2)F_{d+1} + 2F_{d+1} + (d-2)F_{d+1} \notag \\
 & = -F_d + F_{d+1} + F_{d+1} \notag \\
 & = F_{d-1} + F_{d+1} \notag \\
 & = L_d, \notag
 \end{align}
where the equality at line \eqref{recipr2} follows by~\cite[p. 3, Lemma 1]{ss22}.

\end{example}

\section{A $q$-analogue of $r(n)$}

As a final observation, we note that the expressions $(a_i)_{m_{i}}/{m_i!}$ appearing in Proposition \ref{thm1} have a natural $q$-analogue
$$ \lim_{q \to 1} \qbin{a_i - 1 + m_i}{m_i}{q} = \frac{(a_i)_{m_i}}{m_i!},$$
so long as $a_i\geq 0$, where
$$ \qbin{n}{k}{q} = \frac{(q;q)_n}{(q;q)_k(q;q)_{n-k}} $$
is the Gaussian binomial coefficient, $0\leq k\leq n$. Then for $|q|<1$, we may define $r_q(0):=1$ and, for $n\geq 1$, 
$$ r_q(n) := \sum_{\lambda \vdash n} \qbin{a_1 - 1 + m_1}{m_1}{q} \qbin{a_2 - 1 + m_2}{m_2}{q} \cdots $$
to yield a $q$-analogue of $r(n)$, satisfying $\lim_{q \to 1} r_q(n) = r(n)$. 

The generating function of $r_q(n)$
may be obtained by attaching $z^n$ to each side of this equation, and summing over all $n \geq 0$. This yields
a sum now over  all the partitions on the right-hand side:
\begin{equation}
   \sum_{n =0}^\infty r_q(n) z^n = \sum_{\lambda\in\mathcal P} z^{\displaystyle \lvert \lambda \rvert}
   \qbin{a_1 - 1 + m_1}{m_1}{q} \qbin{a_2 - 1 + m_2}{m_2}{q} \cdots.
\label{gf_identity_sum_over_partitions}
\end{equation}
This identity leads us to a product generating formula for $r_q(n)$.. 

\begin{theorem}\label{analogue} If $r(n)$ and $a_n$ are defined as in equation~\eqref{ps} such that $a_n\geq 0$ is satisfied for all $n\geq 0$, then we have 
$$ \sum_{n =0}^\infty r_q(n) z^n = \prod_{n =1}^\infty \frac{1}{(z^n; q)_{a_n}}. $$
\end{theorem}

\begin{proof}
Expanding the product
\begin{equation}\label{expand} \prod_{n=1}^\infty \left(1 + \qbin{a_n - 1 + 1}{1}q z^n + \qbin{a_n - 1 + 2}{2}q z^{2n} + \qbin{a_n - 1 + 3}{3}q z^{3n} + \cdots \right),\end{equation}
we obtain the right hand side of (\ref{gf_identity_sum_over_partitions}). Now, using the well-known $q$-series identity
$$\sum_{k=0}^\infty (qz)^k \qbin{a - 1 + k}{k}q  = \frac{1}{(qz;q)_{a}} $$
with $a\geq 0$, we make the substitutions $z \mapsto z^{a_n}q^{-1}$ and $a \mapsto a_n\geq 0$,  take the product over $n\geq 1$, then compare with \eqref{expand}, to arrive at the result.
\end{proof}

\section{Open questions}
We close with some open questions and suggestions for future related work.
\begin{enumerate}
  \item Can one impose conditions on the $r(n)$ that would guarantee the $a_n$ are nonnegative, so that the $r(n)$ might be explicitly counting something?
 \item Can one impose conditions on the $r(n)$ so that the $a_n$ are periodic with respect to a fixed modulus, as in the product side of Rogers--Ramanujan type identities?
 \item Can one impose conditions on the $a_n$ that  guarantee the $r(n)$ exponentially, subexponentially, or at specified rates?
  \item Is there a natural combinatorial interpretation for the $q$-analogue   $r_q(n)$?
  \item Does Theorem \ref{analogue} (or a generalization) hold for $a_n<0$ under the definition of the $q$-Pochhammer symbol for negative indices, viz. $(a;q)_{-n}:=(aq^{-n};q)_n^{-1}$?  
\end{enumerate}

\section*{Acknowledgment}
The authors wish to thank the anonymous referee for helpful suggestions that guided us to
improve the manuscript.


\begin{thebibliography}{19}
  \bibitem{gea76} Andrews, G.E.: The Theory of Partitions, Addison--Wesley, Reading MA (1976)
 
  \bibitem{ap07} Andrews, G.E., Paule, P.: MacMahon's partition analysis XI: Broken
  diamonds and modular forms, Acta Arith 126, 281--294 (2007)

  \bibitem{b99} Bressoud, D.M.: Proofs and Confirmations: the Story of the Alternating
  Sign Matrix Conjecture, Cambridge University Press, Cambridge (1999)
  
  \bibitem{bbbf24} Bridges, W., Brindle, B.,  Bringmann, K., Franke, J.:
  Asymptotic expansions for partitions generated by infinite products,
 {Mathematische Annalen} 390, 2593--2632 (2024)
 
 \bibitem{cl04} Corteel, S., Lovejoy, J.: Overpartitions, Trans. Amer. Math. Soc. 356, 1623--1635
 (2004)

  \bibitem{f88} Fine, N.J.: Basic Hypergeometric Series and Applications, 
  Mathematical Surveys and Monographs, no. 27, Amer. Math. Soc., Providence, RI (1988)
  
  \bibitem{fg20} Garvan, F.G.:  An updated q-product tutorial for Maple,
  \texttt{https://qseries.org/fgarvan/qmaple/qseries/doc/qseriesdoc.pdf}
 
  \bibitem{kr15} Kanade, S, Russell, M.: 
  IdentityFinder and some new identities of Rogers--Ramanujan type,
  Experimental Mathematics 24, 419--423 (2015)

  \bibitem{gm54} G. Meinardus, Asymptotosche aussagen \"{u}ber Partitionen, Math. Z. 
  59, 388–398 (1954)

  \bibitem{oeisA006206} The On-Line Encyclopedia of Integer Sequences, published electronically at \texttt{https://oeis.org} Sequence A006206 (2024)

  \bibitem{oeisA059966} The On-Line Encyclopedia of Integer Sequences, published electronically at \texttt{https://oeis.org} Sequence A059966 (2024)


\bibitem{Schneider_PhD} Schneider, R. {Eulerian Series, 
Zeta Functions and the Arithmetic of Partitions}, Ph.D. dissertation, Emory University (2018)

 
\bibitem{ss22} Schneider, R. Sills, A.V.: Combinatorial formulas for arithmetic
  density,  Integers 22,  paper \#A63,  7 pp. (2022)
  
 \bibitem{ht11} Todt, H.: {Asymptotics of Partition Functions}, Ph.D. dissertation, Pennsylvania State 
 University (2011) 
  
\end{thebibliography}
\end{document}